\documentclass[12pt]{amsart}

\usepackage{fullpage,url,amssymb,amsmath,amsthm,amsfonts,mathrsfs,hyperref,enumitem}
\usepackage[alphabetic,lite]{amsrefs}
\usepackage[all, cmtip]{xy} % for complicated commutative diagrams

\usepackage{color}

\DeclareFontEncoding{OT2}{}{} % to enable usage of cyrillic fonts
\newcommand{\textcyr}[1]{{\fontencoding{OT2}\fontfamily{wncyr}\fontseries{m}\fontshape{n}\selectfont #1}}
\newcommand{\Sha}{{\mbox{\textcyr{Sh}}}}

\newcommand{\Z}{{\mathbb Z}}
\newcommand{\Q}{{\mathbb Q}}
\newcommand{\R}{{\mathbb R}}
\newcommand{\F}{{\mathbb F}}
\newcommand{\Kbar}{{\overline{K}}}

\newcommand{\To}{\longrightarrow}
\newcommand{\defi}[1]{\textsf{#1}} % for defined terms

\DeclareMathOperator{\res}{res}
\DeclareMathOperator{\Gal}{Gal}
\DeclareMathOperator{\HH}{H}

\newtheorem*{thmI}{Theorem}
\newtheorem{Theorem}{Theorem}[section]
\newtheorem{Lemma}[Theorem]{Lemma}
\newtheorem{Proposition}[Theorem]{Proposition}
\newtheorem{Corollary}[Theorem]{Corollary}

\newtheorem{Remark}[Theorem]{Remark}

\numberwithin{equation}{section}
\numberwithin{table}{section}

\begin{document}%%%%%%%%%%%%%%%%%%%%%%%%
%%%%%%%%%%%%%%%%%%%%%%%%%%%%%%%%%
%%%%%%%%%%%%%%%%%%%%%%%%%%%%%%%%%

\begin{abstract}
For every prime power $p^n$ with $p = 2$ or $3$ and $n\ge 2$ we give an example of an elliptic curve over $\Q$ containing a rational point which is locally divisible by $p^n$ but is not divisible by $p^n$. For these same prime powers we construct examples showing that the analogous local-global principle for divisibility in the Weil-Ch\^atelet group can also fail.
\end{abstract}

\title{On the local-global principle for divisibility in the cohomology of elliptic curves}

\author{Brendan Creutz}
\address{School of Mathematics and Statistics, The University of Sydney, NSW 2006, Australia}
\email{brendan.creutz@maths.usyd.edu.au}
\urladdr{http://magma.maths.usyd.edu.au/\~{}bcreutz}
%\date{ May 24 2013}

\maketitle

%%%%%%%%%%%%%%
\section{Introduction}%
%%%%%%%%%%%%%%
	Let $G$ be a connected commutative algebraic group over a number field $k$, and let $n$ and $r$ be nonnegative integers. An element $\rho$ in the Galois cohomology group $\HH^r(k,G) := \HH^r(\Gal(\overline{k}/k),G(\overline{k}))$ is \defi{divisible by $n$} if there exists $\rho' \in \HH^r(k,G)$ such that $n\rho'=\rho$. We say $\rho$ is \defi{locally divisible by $n$} if, for all primes $v$ of $k$, there exists $\rho'_v \in \HH^r(k_v,G)$ such that $n\rho'_v = \res_v(\rho)$. It is natural to ask whether every element locally divisible by $n$ is necessarily divisible by $n$. When the answer is yes, we say the \defi{local-global principle for divisibility by $n$} holds.
	
	For $r=0$ and $G = \mathbb G_m$, the answer is given by the Grunwald-Wang theorem (see~\cite{CoNF}*{IX.1}); the local-global principle for divisibility by $n$ holds, except possibly when $8$ divides $n$. The case $r = 1$ and $G = \mathbb G_m$ is trivial in light of Hilbert's theorem 90. For $r \ge 2$ and general $G$, a result of Tate implies that the local-global principle for divisibility by $n$ always holds (see Theorem~\ref{thm:obst} below). 
	
	A study of the problem for $r=0$ and general $G$ was initiated by Dvornicich and Zannier in~\cite{DZ1}, with particular focus on elliptic curves in~\cites{DZexamples,DZ2,PRV-2}. For elliptic curves over $\Q$, their results show that the local-global principle for divisibility by a prime power $p^n$ holds for $n =1$ or $p \ge 11$, and they have constructed counterexamples for $p^n = 4$.\footnote{Paladino et al. have recently announced a proof of the local-global principle for powers of $5$ and $7$~\cite{PRV-3}.} For $r = 1$ and $G$ an elliptic curve, the question was in effect raised by Cassels \cite{CasselsIII}*{Problem 1.3}. In particular, he asked whether elements of $\HH^1(k,G)$ that are everywhere locally trivial must be divisible. In response, Tate proved the local-global principle for divisibility by a prime $p$~\cite{CasselsIV}. Cassels' question is considered again in~\cite{Bashmakov}, and recently by \c Ciperiani and Stix \cite{CipStix} who showed that, for elliptic curves over $\Q$, the local-global principle for divisibility by $p^n$ holds for all prime powers with $p \ge 11$. An example showing that it does not hold in general over $\Q$ for any $p^n = 2^n$ with $n \ge 2$ was constructed in~\cite{CreutzWCDiv}. 
	
	In this note we produce examples settling these questions for the remaining undecided powers of the primes $2$ and $3$. We prove the following.
	\begin{thmI}
		Let $n \ge 2$ be an integer, let $p \in \{2,3\}$ and let $r \in \{0,1\}$. Then there exists an elliptic curve $E$ over $\Q$ for which the local-global principle for divisibility by $p^n$ fails in $\HH^r(\Q,E)$.
	\end{thmI}

%%%%%%%%%%%%%%%%%%
\subsection*{Acknowledgements}%
%%%%%%%%%%%%%%%%%%
It is a pleasure to thank Jakob Stix for a number of helpful comments, including pointing out case~\eqref{it:r>1} in Theorem~\ref{thm:obst}.

%%%%%%%%%%%%
\subsection*{Notation}%
%%%%%%%%%%%%
	Throughout the paper $p$ denotes a prime number, $m$ and $n$ are a positive integers, and $r$ is a nonnegative integer. As above, $G$ is a connected commutative algebraic group defined over a number field $k$ with a fixed algebraic closure $\overline{k}$. We will use $K$ to denote a field containing $k$ and use $\overline{K}$ to denote a fixed algebraic closure of $K$ containing $\overline{k}$. For a $\Gal(\overline{k}/k)$-module $M$, let $M^\vee$ denote its Cartier dual and define
	\[
		\Sha^r(k,M) := \ker\left(\HH^r(k,M) \stackrel{\prod \res_v}\To \prod_v\HH^r(k_v,M)\right)\,,
	\]
	the product running over all primes of $k$.

%%%%%%%%%%%%%%%%%%%%%%%%%%%%%%%%%%%%
\section{The obstruction to the local-global principle for divisibility}%
%%%%%%%%%%%%%%%%%%%%%%%%%%%%%%%%%%%%
	Because $K$ has characteristic $0$, multiplication by $n$ is a finite \'etale endomorphism of $G$. Hence, for any $r \ge 0$, the short exact sequence of $\Gal(\overline{K}/K)$-modules
	\[
		0 \to G[n] \stackrel{\iota}\to G \stackrel{n}\to G \to 0
	\]
	gives rise to an exact sequence,
	\begin{equation}\label{eq:Kummer1}
		\HH^{r}(K,G[n]) \stackrel{\iota_*}\To \HH^{r}(K,G) \stackrel{n_*}\To \HH^{r}(K,G) \stackrel{\delta_n}\To \HH^{r+1}(K,G[n]) \stackrel{\iota_*}\To \HH^{r+1}(k,G)\,.
	\end{equation}
	
	From this one easily sees that an element $\rho \in \HH^r(k,G)$ is locally divisible by $n$ if and only if $\delta_n(\rho) \in \Sha^{r+1}(k,G[n])$, and that $\rho$ is divisible by $n$ if and only if $\delta_n(\rho) = 0$. In particular, the local-global principle for divisibility by $n$ in $\HH^r(k,G)$ holds whenever $\Sha^{r+1}(k,G[n]) = 0$. Combining this observation with Tate's duality theorems yields the following.
		
	\begin{Theorem}\label{thm:obst}
		Assume any of the following:
		\begin{enumerate}
			\item\label{it:r=0} $r = 0$ and $\Sha^1(k,G[n]) = 0$;
			\item\label{it:r=1} $r = 1$ and $\Sha^1(k,G[n]^\vee) = 0$; or
			\item\label{it:r>1} $r \ge 2$.
		\end{enumerate}
		Then the local-global principle for divisibility by $n$ in $\HH^r(k,G)$ holds.
	\end{Theorem}
	
	\begin{proof}
		As noted above, in each case it suffices to show that $\Sha^{r+1}(k,G[n]) = 0$. Case~\eqref{it:r=0} is trivial, and cases~\eqref{it:r=1} and~\eqref{it:r>1} follow immediately from~\cite{Tate-duality}*{Theorem 3.1}.
	\end{proof}
	
	The following proposition shows that when $G$ is a principally polarized abelian variety, the conditions in the theorem are necessary, at least conjecturally.
	
	\begin{Proposition}\label{prop:tricotomy} 
		Suppose $G$ is an abelian variety with dual $G^\vee$. Then for every $\xi \in \Sha^1(k,G[n])$, exactly one of the following hold:
		\begin{enumerate}
			\item $\xi = 0\,;$
			\item\label{it:x=deltay} $\xi = \delta_n(\rho)$ for some $\rho \in G(k)$ that is locally divisible by $n$, but is not divisible by $n\,;$ or
			\item\label{it:iotax} $\iota_*(\xi) \ne 0$, in which case there either exists $\rho \in \Sha^1(k,G^\vee)$ such that $\rho$ is not divisible by $n$, or $\iota_*(\xi)$ is divisible in $\Sha^1(k,G)$ by all powers of $n$.
		\end{enumerate}
		If $G$ is a principally polarized abelian variety and $\Sha^1(k,G)$ is finite, then the local-global principle for divisibility by $n$ holds in $\HH^r(k,G)$ for every $r \ge 0$ if and only if $\Sha^1(k,G[n])=0$.
	\end{Proposition}
	
	\begin{proof}
		Exactness of~\eqref{eq:Kummer1} implies that the cases in the first statement of the proposition are exhaustive and mutually exclusive. For the claim in case~\eqref{it:iotax} we may apply~\cite{CreutzWCDiv}*{Thm 3}, which states that $\Sha^1(k,G^\vee) \subset n\HH^1(k,G^\vee)$ if and only if the image of $\iota_*:\Sha^1(k,G[n]) \to \Sha^1(k,G)$ is contained in the maximal divisible subgroup of $\Sha^1(k,G)$.
		
		Now suppose $G$ is a principally polarized and that $\Sha^1(k,G)$ is finite. We must prove the equivalence in the second statement. One direction follows from Theorem~\ref{thm:obst} since $G[n] = G[n]^\vee$. The other direction follows from the first statement in the proposition, since finiteness of $\Sha^1(k,G)$ implies that it contains no nontrivial divisible elements as in case~\eqref{it:iotax}.
	\end{proof}
		
	The next lemma formalizes a method for constructing elements of $\Sha^1(k,G[mn])$, for some $m \ge 1$.
	\begin{Lemma}\label{lem:1impliesn}
		Let $m \ge 1$ and let $j : G[n] \subset G[mn]$ be the inclusion map. Suppose $\xi \in \HH^1(k,G[n])$ is such that $\res_v(\xi) \in \delta_n(G(k_v)[m])$, for all primes $v$ of $k$. Then
		\begin{enumerate}
			\item $j_*(\xi) \in \Sha^1(k,G[mn])$;
			\item $j_*(\xi) = 0$ if and only if $\xi \in \delta_n(G(k)[m])$;
			\item if $\xi = \delta_n(\rho)$ for some $\rho \in G(k)$, then $m\rho$ is locally divisible by $mn$; and
			\item if $\xi = \delta_n(\rho)$ for some $\rho \in G(k)$ and $j_*(\xi) \ne 0$, then $m\rho$ is not divisible by $mn$.
		\end{enumerate}
	\end{Lemma}
	
	\begin{proof}
		The connecting homomorphism $G(K)[m] \to \HH^1(k,G[n])$ arising from the short exact sequence
		\[
			0 \to G[n] \stackrel{j}\to G[mn] \stackrel{n}\to G[m] \to 0
		\]
		is the restriction of the $\delta_n$ to $G(K)[m]$. This implies that 
		\[
			\ker \left(j_*:\HH^1(K,G[n]) \to \HH^1(K,G[mn])\right) = \delta_n(G(K)[m])\,,
		\]
		from which the first two statements in the proposition easily follow.
		
		The inclusion $j:G[n] \subset G[mn]$ also induces a commutative diagram
		\begin{equation}\label{eq:Kummer2}
			\xymatrix{
					G(K)[n] \ar@{^{(}->}[r]\ar[d]^j
					&G(K)\ar[r]^n\ar@{=}[d]
					&G(K)\ar[rr]^{\delta_n} \ar[d]^{m}
					&&\HH^1(K,G[n]) \ar[d]^{j_*}\\
					G(K)[mn] \ar@{^{(}->}[r]
					&G(K)\ar[r]^{mn}
					&G(K)\ar[rr]^{\delta_{mn}} 
					&&\HH^1(K,G[mn])\,,
			}
		\end{equation}
		where the rows are the exact sequence~\eqref{eq:Kummer1} with $r = 0$, and the same sequence with $mn$ in place of $n$. From this the last two statements can be deduced easily.
	\end{proof}

%%%%%%%%%%%%%%%%%%%%%
\section{The examples for $p = 2$}%
%%%%%%%%%%%%%%%%%%%%%

	\begin{Proposition}
	\label{prop:HPby2nfailsinH0}
		Let $E$ be the elliptic curve defined by $y^2 = (x+2795)(x-1365)(x-1430)$ and let $P = (341:59136:1) \in E(\Q)$. For every $n \ge 2$, the point $2^{n-1}P$ is locally divisible by $2^n$, but not divisible by $2^n$. In particular, the local-global principle for divisibility by $2^n$ in $E(\Q)$  fails for every $n \ge 2$.
	\end{Proposition}
	
	\begin{Remark}
		This example was constructed by Dvornicich and Zannier who proved the proposition in the case $n = 2$~\cite{DZexamples}*{\S4}. Using Lemma~\ref{lem:1impliesn} their arguments apply to all $n \ge 2$. We include our own proof here since our examples for $p = 3$ will be obtained using a similar, though more involved argument.
	\end{Remark}
	
	\begin{proof}
		Fix the basis $P_1 = (1365:0:1)$, $P_2 = (1430:0:1)$ for $E[2]$.  By \cite{Silverman}*{Proposition X.1.4} the composition of $\delta_2$ with isomorphism $\HH^1(K,E[2]) \simeq \left(K^\times/K^{\times 2}\right)^2$ is given explicitly by
		\[ 
			Q = (x_0,y_0) \longmapsto 
			\begin{cases} 
				(x_0-1365,x_0-1430) & \mbox{if } Q \ne P_1,P_2 \\
				(-1,-65) &  \mbox{if } Q = P_1 \\
				(65,65) & \mbox{if } Q = P_2 \\
				(1,1) & \mbox{if } Q = 0 
			\end{cases}\,.
		\]
		In particular, $\delta_2(P) = (-1,-1)$ and $\delta_2(E(K)[2])$ is generated by $\{ (-1,-65),(65,65) \}$. It follows that $\delta_2(P) \in \delta_2(E(K)[2])$ if and only if  at least one of $65$, $-65$ or $-1$ is a square in $K$. If $K = \Q_v$ for some $v \le \infty$, then one of these is a square. Indeed, $65$ is a square in $\R$ and in $\Q_2$, $-1$ is a square $\Q_5$ and in $\Q_{13}$, and for all other primes $v$ the Legendre symbols satisfy the identity $\left(\frac{-1}{v}\right)\left(\frac{65}{v}\right)=\left(\frac{-65}{v}\right)$. Hence $\xi := \delta_2(P)$ satisfies the hypothesis of  Lemma~\ref{lem:1impliesn} with $(m,n)$ replaced by $(2^{n-1},2)$.
		
		On the other hand, $65$, $-65$ and $-1$ are not squares in $\Q$, and $E(\Q)[2^\infty] = E(\Q)[2]$ (the reduction mod $3$ is nonsingular, so the $2$-primary torsion must inject into the group of $\F_3$-points on the reduced curve. This group has order less than $8$ by Hasse's theorem). So the result follows from Lemma~\ref{lem:1impliesn}.		
	\end{proof}
	
	\begin{Proposition}
		Let $E$ be the elliptic curve defined by $y^2 = x(x+80)(x+205)$. Then $\Sha^1(\Q,E) \not\subset 4\HH^1(\Q,E)$. In particular, the local-global principle for divisibility by $2^n$ in $\HH^1(\Q,E)$ fails for every $n \ge 2$.
	\end{Proposition}
	
	\begin{proof}
		This is \cite{CreutzWCDiv}*{Theorem 5}; we are content to sketch the proof. Much like the previous proof, one uses the explicit description of the map $\delta_2: E(K) \to \HH^2(K,E[2])\simeq \left(K^\times/K^{\times 2}\right)^2$ to show that there is an element $\xi \in \HH^1(\Q,E[2]) \setminus \delta_2(E(\Q))$ which maps into $\delta_2(E(\Q_v))$ everywhere locally. Lemma~\ref{lem:1impliesn} then shows that the image of $\xi$ in $\HH^1(k,E[4])$ falls under case~\eqref{it:iotax} of Proposition~\ref{prop:tricotomy}. This gives the result, since $\Sha^1(\Q,E)[2^\infty]$ is finite (as one can check in multiple ways, with or without the assistance of a computer).			
	\end{proof}

%%%%%%%%%%%%%%%%%%%%%%%%%%
\section{Diagonal cubic curves and $3$-coverings}%
%%%%%%%%%%%%%%%%%%%%%%%%%%
	The examples for $p=2$ were constructed using an explicit description of the map
	\[
		E(K) \stackrel{\delta_2}\To \HH^1(K,E[2]) \simeq \left(K^\times/K^{\times 2}\right)^2\,.
	\]
	Another way to describe the connecting homomorphism is in the language of $n$-coverings. An \defi{$n$-covering} of an elliptic curve $E$ over $K$ is a $K$-form of the multiplication by $n$ map on $E$. In other words, an $n$-covering of $E$ is a morphism $\pi : C \to E$ such that there exists an isomorphism $\psi:E_\Kbar \to C_\Kbar$ of the curves base changed to the algebraic closure $\Kbar$ which satisfies $\pi \circ \psi = n$. We now summarize how this notion can be used to give an interpretation of the group $\HH^1(K,E[n])$. Details may be found in \cite{CFOSS1}*{\S1}. 
	
	An isomorphism of $n$-coverings of $E$ is, by definition, an isomorphism in the category of $E$-schemes. The automorphism group of the $n$-covering $n:E\to E$ can be identified with $E[n]$ acting by translations. By a standard result in Galois cohomology (the twisting principle) the $K$-forms of $n:E\to E$ are parameterized, up to isomorphism by $\HH^1(K,E[n])$. Under this identification the connecting homomorphism $\delta_n$ sends a point $P \in E(K)$ to the isomorphism class of the $n$-covering,
	\[
		\pi_P : E \to E\,,\quad Q \mapsto nQ + P\,.
	\]
	In particular, the isomorphism class of an $n$-covering $\pi:C\to E$ is equal to $\delta_n(P)$ if and only if $P \in \pi(C(K))$.

	Our examples for $p = 3$ will come from elliptic curves of the form $E : x^3 + y^3 + dz^3 = 0$ with distinguished point $(1:-1:0)$, where $d \in \Q^\times$. For these curves we can write down some of the $3$-coverings quite explicitly. According to Selmer, the following lemma goes back to Euler (see \cite{Selmer}*{Theorem 1}).
	
	\begin{Lemma}\label{lem:selmer}
		Let $E : x^3 + y^3 + dz^3 = 0$ and suppose $a,b,c \in \Q^\times$ are such that $abc = d$. Then the curve $C: aX^3 + bY^3 + cZ^3 = 0$ together with the map $\pi : C \to E$ defined by 
			\begin{align*}
				x+y &= 9abcX^3Y^3Z^3\\
				x-y  &= (aX^3-bY^3)(bY^3-cZ^3)(cZ^3-aX^3)\\
				z &= 3(abX^3Y^3 + bcY^3Z^3 + caZ^3X^3)XYZ
			\end{align*}
		is a $3$-covering of $E$.	
	\end{Lemma}
	
	\begin{proof}
		A direct computation verifies that these equations define a nonconstant morphism $\pi : C \to E$, which, by virtue of the fact that $E$ and $C$ are smooth genus $1$ curves, implies that it is finite and \'etale. The map $\psi: E_\Kbar \to C_\Kbar$ defined by
		\begin{equation}
			x = \sqrt[3]{a}X\,, \quad 
			y = \sqrt[3]{b}Y\,, \quad 
			z = \sqrt[3]{c/d}Z
		\end{equation}
		is clearly an isomorphism. It is quite evident that $E[3]$, which is cut out by $xyz = 0$, is mapped by $\pi\circ\psi$ to the identity $(1:-1:0) \in E_\Kbar$. Therefore $\pi\circ\psi$ is an isogeny which factors through multiplication by $3$. Since it has degree $9$ it must in fact be  multiplication by $3$, and so $\pi$ is a $3$-covering.
	\end{proof}
	
	\begin{Lemma}\label{lem:d'cond}
		Suppose $d = 3d'$ and let $\xi \in \HH^1(K,E[3])$ be the class corresponding to the $3$-covering as in Lemma~\ref{lem:selmer} with $C : X^3 + 3Y^3 +d'Z^3 = 0$. Then $\xi \in \delta_3(E(K)[3])$ if any of the following hold:
		\begin{enumerate}
			\item\label{it:3iscube} $3 \in K^{\times 3}\,;$
			\item\label{it:d'iscube} $d' \in K^{\times 3}\,;$
			\item\label{it:3discube} $3d \in K^{\times 3}\,;$
			\item\label{it:zeta9} $d \in K^{\times 3}$ and $K$ contains the $9$th roots of unity; or
			\item\label{it:3zeta} $d \in K^{\times 3}$ and $K$ contains a cube root of unity $\zeta_3$ such that $3\zeta_3 \in K^{\times 3}$.
		\end{enumerate}
	\end{Lemma}
	
	\begin{Corollary}\label{cor:d'cond}
		Suppose $d = 3d'$ and let $\xi \in \HH^1(\Q,E[3])$ be the class of the $3$-covering in Lemma~\ref{lem:d'cond}. Then $\res_v(\xi) \in \delta_3(E(\Q_v)[3])$, for every prime $v \nmid d$.
	\end{Corollary}
	
	\begin{proof}
		Suppose $v \nmid d$ and set $K = \Q_v$. By assumption $d, d', 3$, and $3d$ are units and, since $\Z_v^\times/\Z_v^{\times 3}$ is cyclic, one of them must be a cube. Moreover, if $\Q_v$ does not contain a primitive cube root of unity, then they are all cubes (since $\Z_v^\times/\Z_v^{\times 3}$ is trivial in this case). In light of this, and the first three cases in the lemma, we may assume $d \in \Q_v^{\times 3}$ and that $\Z_v$ contains a primitive cube root of unity $\zeta_3$. If $\zeta_3$ is a cube, then case~\eqref{it:zeta9} of the lemma applies. If $\zeta_3$ is not a cube, then the class of $3$ is contained in the subgroup of $\Q_v^\times/\Q_v^{\times 3}$ generated by $\zeta_3$, in which case~\eqref{it:3zeta} of the lemma applies. This establishes the corollary.
	\end{proof}
	
	\begin{proof}[Proof of Lemma~\ref{lem:d'cond}]
		By the discussion at the beginning of this section, it suffices to show that in each of these cases there is a $K$-rational point on $C$ which maps to a $3$-torsion point on $E$.\footnote{The points given below were found with the assistance of the {\tt Magma} computer algebra system described in~\cite{MAGMA}. A {\tt Magma} script verifying the claims here can be found in the source file of the {\tt arXiv} distribution of this article.} The $3$-torsion points are the intersections of $E$ with the hyperplanes defined by $x = 0$, $y = 0$ and $z = 0$. In the first three cases (resp.) the points
		\[
			(-\sqrt[3]{3}:1:0)\,,\quad (-\sqrt[3]{d'}:0:1)\,,\text{ and}\quad (0:-\sqrt[3]{3d}:3)
		\]
		are defined over $K$, and the explicit formula for $\pi$ given in Lemma~\ref{lem:selmer} shows that they map to $(1:-1:0) \in E(K)[3]$. 
		
		In case~\eqref{it:zeta9} $K$ Contains a primitive $9$th root of unity $\zeta_9$ and a cube root $\sqrt[3]{d}$ of $d$. Then
			\[ \left(
				(2\zeta_9^5 + \zeta_9^4 + \zeta_9^2 + 2\zeta_9)\sqrt[3]{d} : 
				(-\zeta_9^3 + \zeta_9^2 + \zeta_9 - 1)\sqrt[3]{d} :
				-3\right) \in C(K)\,,
			\]
			and one can check that it maps under $\pi$ to the point $(0:-\sqrt[3]{d}:1)$. In case~\eqref{it:3zeta} $K$ contains cube roots $\sqrt[3]{d}$ and $\beta = \sqrt[3]{3\zeta_3}$, where $\zeta_3$ is a cube root of unity. One may check that	$(\beta^2\sqrt[3]{d} : \beta\sqrt[3]{d} : -3) \in C(K)$, and that this point maps under $\pi$ to the point $(\zeta^2:-1:0)$.
	\end{proof}

%%%%%%%%%%%%%%%%%%%%%
\section{The examples for $p = 3$}%
%%%%%%%%%%%%%%%%%%%%%

	\begin{Proposition}
		Let $E : x^3 + y^3 + 30z^3 = 0$ be the elliptic curve over $\Q$ with distinguished point $P_0 = (1:-1:0)$, and let $P = (1523698559 : -2736572309 : 826803945 ) \in E(\Q)$. For every $n \ge 2$, $3^{n-1}P$ is locally divisible by $3^{n}$, but not divisible by $3^{n}$. In particular, the local-global principle for divisibility by $3^n$ in $E(\Q)$ fails for every $n \ge 2$.
	\end{Proposition}
	
	\begin{proof}
		Let $C : X^3 + 3Y^3 + 10Z^3$ be the $3$-covering of $E$ as in Lemma~\ref{lem:selmer}, and let $\xi \in \HH^1(\Q,E[3])$ be the corresponding cohomology class. One may check that the point $Q = (-11:3:5) \in C(\Q)$ maps to $P$. Thus $\xi = \delta_3(P)$. By Corollary~\ref{cor:d'cond}, $\res_v(\xi) \in \delta_3(E(\Q_v)[3])$ for all primes $v \nmid 30$. Also, since $10 \in \Q_3^{\times 3}$ and $3$ is a cube in both $\Q_2$ and $\Q_5$ the first two cases of Lemma~\ref{lem:d'cond} show that $\res_v(\xi) \in \delta_3(E(\Q_v)[3])$ also for $v \mid 30$. On the other hand, $\xi \ne 0$ because $C(\Q)$ does not contain a point lying on the subscheme defined by $XYZ=0$. Since, $E(\Q)[3] = 0$ the result follows by applying Lemma~\ref{lem:1impliesn}.
	\end{proof}
	
	\begin{Remark}
		For any $d \in \{51, 132, 159, 213, 219, 246, 267, 321, 348, 402, 435 \}$ the same argument applies, giving more examples where the local-global principle for divisibilitiy by $3^n$ in $E(\Q)$ fails for all $n \ge 2$.
	\end{Remark}
	
	\begin{Proposition}
		Let $d \in \{ 138,165,300,354 \}$ and let $E : x^3 + y^3 + dz^3 = 0$ be the elliptic curve over $\Q$ with distinguished point $P_0 = (1:-1:0)$. Then $\Sha^1(\Q,E) \not\subset 9\HH^1(\Q,E)$. In particular, the local-global principle for divisibility by $3^n$ in $\HH^1(\Q,E)$ fails for every $n \ge 2$.
	\end{Proposition}
	
	\begin{proof}
		Set $d' = d/3$. Let $C : X^3 + 3Y^3 + d'Z^3$ be the $3$-covering of $E$ as in Lemma~\ref{lem:selmer}, and let $\xi \in \HH^1(\Q,E[3])$ be the corresponding cohomology class. In all cases one easily checks that $d' \in \Q_3^{\times 3}$ and that $3 \in \Q_v^{\times 3}$ for all $v \mid d'$. So using the first two cases of Lemma~\ref{lem:d'cond} and Corollary~\ref{cor:d'cond} we see that $\res_v(\xi) \in \delta_3(E(\Q_v)[3])$ for every prime $v$. Then, by Lemma~\ref{lem:1impliesn}, the image of $\xi$ in $\HH^1(\Q,E[9])$ lies in $\Sha^1(\Q,E[9])$. 
		
		For these values of $d$, Selmer showed that $E(\Q) = \{ (1:-1:0)\}$ and $C(\Q) = \emptyset$ \cite{Selmer}*{Theorem IX and Table 4b}. The latter implies that the image of $\xi$ in $\Sha^1(\Q,E[3^n])$ is nontrivial for every $n \ge 2$. Moreover, Selmer's proof shows that $3\Sha^1(\Q,E)[3^\infty] = 0$. In particular $\Sha^1(\Q,E)[3^\infty]$ contains no nontrivial infinitely divisible elements. Thus we are in case~\eqref{it:iotax} of Proposition~\ref{prop:tricotomy}, and conclude that there exists some element of $\Sha^1(\Q,E)$ which is not divisible by $9$ in $\HH^1(\Q,E)$.
	\end{proof}
	
	\begin{Remark}
		The argument in the proof above shows that $C \in \Sha^1(\Q,E)$, but does not show that $C \notin 9\HH^1(\Q,E)$. Rather, the elements of $\Sha^1(\Q,E)$ which are proven not to be divisible by $9$ in $\HH^1(\Q,E)$ are those that are not orthogonal to $C$ with respect to the Cassels-Tate pairing. See \cite{CreutzWCDiv}*{Theorem 4}.
	\end{Remark}

%%%%%%%%%%%%%%%%%%%%%%%%%%%%%%%%%%%%%%%%%%%%%%%%%%%%%%%%%%%%%%%%%
%%%%%%%%%%%%%%%%%%				Bibliography			%%%%%%%%%%%%%%%%%%%%%%%%
%%%%%%%%%%%%%%%%%%%%%%%%%%%%%%%%%%%%%%%%%%%%%%%%%%%%%%%%%%%%%%%%%

\end{document}